\documentclass[11pt]{article}
\usepackage{amsmath,amsfonts,amsthm,amssymb, mathtools, MnSymbol, enumitem}
\usepackage{mathrsfs, graphicx,color,latexsym, tikz, calc,booktabs}
\usepackage{caption}
\usepackage{makecell}
\usepackage{indentfirst}
\usepackage{amsmath}
\usepackage{setspace}
\usepackage{multirow,upgreek}
\usepackage{color}
\usetikzlibrary{shadows}
\usetikzlibrary{patterns,arrows,decorations.pathreplacing}
\newtheorem{thm}{Theorem}[section]
\newtheorem{lem}[thm]{Lemma}
\newtheorem{cor}[thm]{Corollary}
\newtheorem{re}[thm]{Remark}
\newtheorem{Def}[thm]{Definition}
\newtheorem{prop}[thm]{Proposition}

\newcommand{\field}[1]{\mathbb{#1}}

\newcommand{\N}{\field{N}}
\newtheorem{examplex}{Example}

\usepackage{etoolbox}
\let\bbordermatrix\bordermatrix
\patchcmd{\bbordermatrix}{8.75}{4.75}{}{}
\patchcmd{\bbordermatrix}{\left(}{\left(}{}{}
\patchcmd{\bbordermatrix}{\right)}{\right)}{}{}

\patchcmd{\bbordermatrix}{\begingroup}{\begingroup\openup1\jot}{}{}
\makeatletter
\patchcmd{\bbordermatrix}
  {\vcenter{\kern-\ht\@ne\unvbox\z@\kern-\baselineskip}}
  {\vcenter{\kern-\ht\@ne\unvbox\z@\kern-\baselineskip\kern2pt}}
  {}{}
\makeatother

\title{\bf Inertia and spectral symmetry of eccentricity matrices of some clique trees }
\author{
{\small  Xiaohong Li$^a$, \ \ Jianfeng Wang$^{a,}$\footnote{Corresponding author.
\newline{\it \hspace*{5mm}Email addresses:} x.h.li@aliyun.com(X.H. Li), jfwang@sdut.edu.cn (J.F.Wang), mbrunett@unina.it (M. Brunetti)}, \;\; Maurizio Brunetti$^b$}\\[2mm]
\footnotesize $^a$School of Mathematics and Statistics, Shandong University of Technology, Zibo 255049, China\\
\footnotesize $^b$ Department of Mathematics and Applications, University Federico II', Naples, Italy }
\date{ }
\date{}

\begin{document}
\maketitle
	
\setcounter{page}{1}

\begin{abstract}
The eccentricity matrix $\mathcal E(G)$ of a connected graph $G$ is obtained from the distance matrix of $G$ by leaving unchanged the largest nonzero entries in each row and each column, and replacing the remaining ones with zeros.  In this paper, we consider the set $\mathcal C  \mathcal T$ of clique trees whose blocks have at most two cut-vertices \textcolor{blue}{of the clique tree}. After proving the irreducibility of the eccentricity matrix of a clique tree in $\mathcal C \mathcal T$ and finding its inertia indices, we show that every graph in $\mathcal C  \mathcal T$  with more than $4$ vertices and odd diameter has two positive and two negative
$\mathcal E$-eigenvalues. Positive $\mathcal E$-eigenvalues and negative $\mathcal E$-eigenvalues turn out to be equal in number even for graphs in $\mathcal C \mathcal T$ with even diameter; that shared cardinality also counts the \textcolor{blue}{`diametrally distinguished'} vertices. Finally, we prove that the spectrum of the eccentricity matrix of a clique tree $G$ in $\mathcal C \mathcal T$ is symmetric with respect to the origin if and only if $G$ has an odd diameter and exactly two adjacent central vertices. \\


\end{abstract}
\section{Introduction}
\noindent
Let $G =(V_G, E_G)$  be a graph with order $n=|V_G|$ and size $m=|E_G|$. Each vertex adjacent to $v \in V_G$ is said to be a {\em neighbour of $v$}. The set of all the neighbours of $v$  is denoted by $N_G(v)$.  The vertex-degree $d_G(v)$ of $v$ gives the cardinality of $N_G(v)$.
The distance between two vertices $u$ and $v$, i.e.\ the minimum length of the paths joining them, is denoted by  $d_{G}(u,v)$ (the subscript  will be omitted when the context makes it clear).
 Let $D(G)=(d_{uv})$ be the distance matrix of $G$, where $d_{uv}=d_{G}(u, v)$. The eccentricity $e_G(u)$ of the vertex $u\in V_G$ in $G$ is given by $e_G(u)=\max\{d_G(u, v) \mid v\in V_G\}$. By definition, the radius and the diameter of $G$ are respectively the number
${\rm rad}(G):=\min\{e_G(u) \mid u\in V_G\}$ and ${\rm diam}(G):=\max\{e_G(u) \mid u\in V_G\}$. A {\em diametral path} is a shortest path between any pair of vertices $u$ and $v$ such that
$d_G(u,v) = {\rm diam}(G)$. The center $C(G)$ of a graph $G$ consists of those vertices whose eccentricity is equal to the radius. Each element in $C(G)$ is said to be a {\em central vertex} or simply a {\em center} of $G$.

The matrix $\mathcal E(G) = (\epsilon_{uv})$, where
\[{\bf \epsilon}_{uv} = \begin{cases} \, d_{G}(u,v)  \quad \text{if $d_{G}(u,v)=min{\{e_G(u),e_G(v)}\}$,}\\
  					       \, 0 \ \ \ \ \ \ \  \ \ \  \quad \text{otherwise,}
\end{cases}
\]
is known as the {\em eccentricity matrix}  of $G$ (see for instance \cite{MK, JM, WB}). Some authors referred to it as the $D_{MAX}$-matrix \cite{MR,MRA}.
 The matrix $\mathcal E(G)$ can be obtained from the distance matrix $D(G)$ by retaining the largest distances in each row and each column and replacing the remaining entries with zeros.
The eccentricity  matrix is symmetric; therefore, the $\mathcal E$-eigenvalues, i.e.\ the eigenvalues of $\mathcal E(G)$, are all real.

 The $\mathcal E$-spectrum of $G$ can be written as $${\rm Spec}_\mathcal E(G) =\{\xi_{1}^{(m_1)},\xi_{2}^{(m_2)},\dots,\xi_{k}^{(m_k)}\},$$ where
$\xi_{1} > \xi_{2} >\cdots > \xi_{k}$ are  the distinct $\mathcal E$-eigenvalues of $G$ and $m_{i}$ is the algebraic multiplicity of $\xi_{i}$ for $1 \leqslant i \leqslant k$. By the Perron-Frobenius Theorem, the $\mathcal E$-spectral radius  $\rho(\mathcal E(G))$, defined as the largest modulus  of the $\mathcal E$-eigenvalues, is attained by the largest $\mathcal E$-eigenvalue.

After ten years of investigation, the \textcolor{blue}{article} on the eccentricity matrix, its spectral properties and some related chemical indices is quite rich.  \textcolor{blue}{The authors} proved that the eccentricity matrices  of trees are irreducible  {\rm \cite{JM}}. Many nonisomorphic $\mathcal E$-cospectral pairs have been found in \cite{WB},  where the reader can also find a characterization of graphs having exactly two distinct $\mathcal E$-eigenvalues. Mahato et al.\ computed the $\mathcal E$-spectra of many graphs in \cite{IM}. Wei et al. {\rm \cite{WW}} determined the $n$-vertex trees with minimum $\mathcal E$-spectral radius, also identifying  which graphs have the minimal $\mathcal E$-spectral radius among  trees with given order and diameter. Qiu et al. {\rm \cite{ZZ}} studied the $\mathcal E$-spectrum of threshold graphs.  $\mathcal E$-energy has been studied in \cite{LX, JG, AKP}. In particular, in  \cite{LX} Lei et al.\ characterized the graphs whose second least $\mathcal E$-eigenvalue is larger than $-\sqrt{15-\sqrt{193}}$. Applications of the eccentricity matrix to molecular descriptors and Chemical graph theory can be found in {\rm \cite{JX, JM}}.

In the last few years, there has been a growing interest in the distance spectral properties of clique trees (the definition is recalled in Section 2).  Extremal problems concerning the distance energy of clique trees
have been solved in \cite{JGZ,LLL}. Moreover, Zhang et al. {\rm \cite{XJ}} determined the graphs with maximum and minimum distance Laplacian spectral radius among all clique trees with $n$ vertices and $k$ cliques.

 In the wake of \cite{JGZ,LLL,XJ}, we study in this paper some additional $\mathcal E$-spectral properties of a large family of clique trees including all ordinary trees. More precisely, we consider the class $\mathcal C \mathcal T$ of clique trees whose blocks have at most two cut-vertices \textcolor{blue}{of the clique tree}, compute the $\mathcal E$-inertia of all clique trees in $\mathcal C \mathcal T$  (Theorems \ref{3.5} and \ref{3.6}) and find the structural conditions ensuring $\mathcal E$-spectral symmetry (see Theorems \ref{4.1} and \ref{4.5}). Recall that the inertia ${\rm In}(M)$ of a symmetric matrix $M$ is the triple $(n_{+}(M), n_{-}(M), n_{0}(M))$, where $n_{+}(M)$, $n_{-}(M)$ and $n_{0}(M)$ denote the number of positive, negative and zero eigenvalues of $M$, respectively. One of the reasons why we focus on clique trees in $\mathcal C \mathcal T$ is the following structural property: every clique tree $G$ in $ \mathcal C \mathcal T$ surely contains an induced tree $T_G$ such that (i) ${\rm diam}(G) = {\rm diam}(T_G)$, (ii) the centers of $T_G$ are also centers of $G$, and (iii) the edges in $E_{T_G}$ are in one-to-one correspondence with the blocks of $G$ (see Section 2).

For our arguments, we borrow some ideas from \cite{IM}, where the $\mathcal E$-inertias of paths and lollipop graphs are computed,  and from the papers \cite{MK, AKP},  containing results on the $\mathcal E$-inertia of trees and of coalescences of graphs respectively. We explicitly note that the coalescences of complete graphs considered in \cite{AKP} are all in $\mathcal C\mathcal T$.

\section{Preliminaries and basic tools}
We start by recalling one of the most basic result of spectral matrix theory.
\begin{thm}\label{2.2}{\rm \cite[p.~53]{RHC}} Let $A$ be an $n\times n$ real matrix and let $E_{k}(A)$ denote the sum of its principal minors of size $k$. Then, the characteristic polynomial of $A$ is given by
\[ \varphi(\lambda)=\lambda^{n}-E_{1}(A)\lambda^{n-1} + \cdots + (-1)^{n-1}E_{n-1}(A)\lambda + (-1)^{n}E_{n}(A).
\]
\end{thm}
The next well-known result is the celebrated Cauchy's Interlacing Theorem specialized to the context of symmetric matrices.
\begin{thm}\label{2.1}{\rm \cite[Corollary 2.5.2]{BH}} Let $A$ be an $n\times n$ symmetric real matrix with  eigenvalues $\lambda_{1}\geqslant \dots \geqslant \lambda_{n}$, and let $B$ be an $m \times m$ principal submatrix of $A$ having eigenvalues $\mu_{1}\geqslant \dots \geqslant\mu_{m}$. Then, the eigenvalues of $B$ interlace those of $A$, i.e.
\[ \lambda_{i}\geqslant \mu_{i}\geqslant \lambda_{n-m+i}, \qquad \text{for $1 \leqslant i \leqslant m$.}
\]
\end{thm}

According to \cite[Definitions~6.2.21 and~6.2.22]{RHC}, a matrix $M$ is said to be
irreducible if it is not permutationally similar to an upper triangular block matrix. In other words $M$ is reducible if there
exists a permutation matrix $P$ such that $M= P^\top \begin{pmatrix} A & B \\ O & C\end{pmatrix} P$,
where $A$ and $C$ are both square blocks with at least one row. The following result has been already mentioned in Section 1.
\begin{thm}\label{2.3}{\rm \cite[Theorem2.1]{JM}} The eccentricity matrix $\mathcal E(T)$ of a tree is irreducible.
\end{thm}
It is worthy to recall that there are many connected graphs whose eccentricity matrix is reducible. This happens, for instance, when the graph is complete bipartite of type $K_{a,b}$ with $a \geqslant b \geqslant 2$ \cite{JMF}.

For each symmetric matrix $A=(a_{ij})$, we consider the {\rm indicator matrix} $M(A)=(\mu_{ij})$, where $\mu_{ij}=1$ if $a_{ij}\neq 0$, and    $\mu_{ij}=0$ if $a_{ij}= 0$, and denote by $\Gamma(A)$  \textcolor{blue}{the graph} having $M(A)$ as adjacency matrix. The following theorem can be deduced from \cite[Theorem~6.2.24]{RHC}.
\begin{thm}\label{2.4} Let $A=(a_{ij})$ be the nonnegative symmetric matrix with order $n >1$. Then, $A$ is irreducible if and only if the graph $\Gamma (A)$ is  connected.
\end{thm}
We now recall the definition of a clique tree as given in \cite{JGZ,LLL}. We warn the reader that the same expression also appears in literature with a different meaning  (like in \cite{JB, KM}).

 Let $G$ be a connected graph. A subset $X \subset V_G$ is said to be a cut-vertex set for $G$ if $G\setminus X$ is not connected. If a cut-vertex set $X$ is a singleton, its only element is said to be a {\em cut-vertex} for $G$. A block $B$ of $G$ is a maximal connected subgraph of $G$ without cut-vertices (for $B$).
As noted, for instance, in \cite{D}, the intersection of two blocks of $G$, if nonempty, consists of a single cut-vertex of $G$.

A clique tree is by definition a connected graph whose blocks are all cliques, i.e.\ complete graphs.
Obviously, every tree is a clique tree.  On the contrary, a clique tree is acyclic if and only if each of its blocks contains at most two vertices.  Clique trees are also known as block graphs, and there exists several structural characterizations of them (see \cite{Band, How}). For instance, it can be proved that $G$ is a clique tree if and only if is a diamond-free chordal graph (see \cite[Proposition~2.1]{Band}), i.e.\ no induced subgraph of $G$ is a cycle with more than $3$ vertices or a complete graph $K_4$ minus one edge.

In Section 1 we have denoted by $\mathcal C \mathcal T$ the class of clique trees whose blocks have at most two cut-vertices \textcolor{blue}{of the clique tree}. All clique trees with diameter $1$ or $2$ are in $\mathcal C \mathcal T$. The graph obtained by attaching a pendant edge to each vertex of a triangle represents the clique tree not belonging to $\mathcal C \mathcal T$ which is minimal with respect to both order and diameter.

Let now $G$ be a clique  \textcolor{blue}{tree} in $\mathcal C \mathcal T$. We say that a block $B$ of $G$ is a {\em leaf-block} (resp. {\em bridge-block}) if $B$ contains only one cut-vertex (resp. two cut-vertices) of $G$. If $G$ is a complete graph, we consider its only block, i.e.\ $G$ itself, as a leaf-block.
The subset of $\mathcal C \mathcal T$ containing clique trees whose diameter is at least $2$ is denoted by $\mathcal C \mathcal T^{\geqslant 2}$. Clearly,  $\mathcal C \mathcal T^{\geqslant 2}$ is obtained from $\mathcal C \mathcal T$ by removing the complete graphs.
\begin{lem}\label{lzero} Every clique tree in $\mathcal C \mathcal T^{\geqslant 2}$ has at least two leaf-blocks.
\end{lem}
\begin{proof}
 \textcolor{blue}{Let $G$ be a graph in  $\mathcal C \mathcal T^{\geqslant 2}$ and} $B_1, \dots, B_t$ be the pairwise distict blocks of $G$. If $t=2$, then $G$  is the coalescence of two complete graphs, which clearly  are two leaf-blocks for $G$.

Let now $t \geqslant 3$. We  \textcolor{blue}{model} the argument used to prove that every tree has at least two leaves.
We first show that $G$ surely has one leaf-block.
 If $B_1$ is a leaf-block we are done. Otherwise, let $u$ be one the its two cut-vertices. Surely there exists an $s \not=1$ such that, for a suitable $v\in B_s$ we have
$e_G(u)=d(u,v)$. Since there are no cycles in $G$ intersecting more than one block (see, for instance \cite[Lemma~3.1.2(i)]{D}) and the set $\{ w \in V_G \mid d(u,w) > d(u,v) \}$ is empty, $B_s$ is necessarily a leaf-block of $G$.

Now we can use an inductive argument on the number of blocks, noticing that the graph $G'$ obtained from $G$ by removing all the noncut-vertices of the leaf-block $B_s$ is a clique tree, whose blocks are precisely the $t-1$ remaining blocks of $G$. By the induction hypothesis, at least two blocks of $G'$, say $B_{i_1}$ and $B_{i_2}$, are leaf-blocks. Once again, the absence of a cycle intersecting more than one block shows that at least one set between $\{B_{i_1}, B_{i_2}\}$, $\{B_{i_1}, B_{s}\}$ $\{B_{i_2}, B_{i_s}\}$ only contains leaf-blocks.
\end{proof}

It is not restrictive to assume that the pairwise distinct blocks $B_1, \dots, B_t$ of $G$ all have at least two vertices.
After choosing an ordered vertex-labelling for $G \in \mathcal C \mathcal T^{\geqslant 2}$, say $\{v_1, \dots, v_n\}$, we denote by $\mathcal C_G$ the set of cut-vertices of $G$, and by $\mathcal S_G$ the subset of $V_G$ obtained by picking in each leaf-block the noncut-vertex with minimal label.
As it is customary, we denote by $G[\mathcal C_G \cup \mathcal S_G]$ the subgraph of $G$ induced by the vertices in $\mathcal C_G \cup \mathcal S_G$.

For each $G \in \mathcal C \mathcal T$ we now set
\[ T_G := \begin{cases} K_2 \qquad \quad \quad \;\;\;  \text{ if $G=K_n$,}\\
 G[\mathcal C_G \cup \mathcal S_G] \quad \text{ if $G \in \mathcal C \mathcal T^{\geqslant 2}$.}
\end{cases}
\]

From the definition and \cite[Lemma~3.1.2(i)]{D} we immediately arrive at the following result.
\begin{prop}\label{puno} For each $G \in \mathcal C \mathcal T$, the graph $T_G$ is a tree. Moreover, ${\rm diam} (G) = {\rm diam} (T_G)$.
\end{prop}
The nonpendant vertices of $T_G$ are precisely the cut-vertices of $G$, and there is a one-to-one corespondence between blocks of $G$ and edges of $T_G$.

Note that the isomorphism type of $T_G$ does not really depend on the chosen ordered vertex-labelling, since, for $G \in \mathcal C \mathcal T^{\geqslant 2}$, a shuffling of vertex-labels does not alter the number and the degree of the nonpendant vertices of $T_G$.

Proposition \ref{puno} makes senseful to call $T_G$ the tree {\em associated to $G \in \mathcal C \mathcal T$}.  The clique tree $G$ depicted in Fig.\ 1 has seven blocks. The leaf-blocks among them are those containing $v_1$, $v_7$, $v_9$, $v_{12}$ and $v_{13}$; the homonymous vertices in the associated tree $T_G$ are precisely those of vertex-degree $1$.

Clique trees in $\mathcal C \mathcal T$ such that the associated tree is a path have been considered in \cite{JGZ,LLL}, where they are called {\em clique paths}. In \cite{LLL} the authors also studied the {\em clique stars}, i.e.\  clique trees in $\mathcal C \mathcal T$ whose associated tree is a star (see {\rm \cite[Fig.~1]{LLL}}).

\begin{figure}
\begin{center}
{\begin{tikzpicture}[vertex1_style/.style={circle,draw,minimum size=0.14 cm,inner sep=0pt, fill=black},vertex2_style/.style={circle,draw,minimum size=0.07 cm,inner sep=0pt, fill=black}, nonterminal/.style={
rectangle,
minimum size=2mm,
thin,
draw=black,
top color=white, 
bottom color=white!50!white!50, 
font=\itshape
}]
   \begin{scope}[xshift=10em]
                 \node[nonterminal] at (6.8,1) {$T_G$};
 \node[vertex1_style, label=above:\small$v_1$] (z1) at  (.8,0.54) {};
 \node[vertex1_style, label=above:\small$v_5$] (z5) at  (1.5,0) {};
  \draw[thick] (z1)--(z5);
 \node[vertex1_style, label=above:\small$v_6$] (z6) at  (2.6,0) {};
 \draw[thick] (z6)--(z5);
\node[vertex1_style, label=below:\small$v_7$] (z7) at  (2.25,-1) {};
\node[vertex1_style, label=right:\small$v_9$] (z9) at  (3.35,-.8) {};
 \draw[thick] (z6)--(z7);
 \draw[thick] (z6)--(z9);
\node[vertex1_style, label=south east:\small$v_{11}$] (z11) at  (3.4,.7) {};
\node[vertex1_style, label=left:\small$v_{12}$] (z12) at  (2.6,1.4) {};
\node[vertex1_style, label=south east:\small$v_{13}$] (z13) at  (4.65,1) {};
 \draw[thick] (z6)--(z11)--(z12);
 \draw[thick] (z11)--(z13);
		
\end{scope}
   \begin{scope}[ xshift=-10em]
              \node[nonterminal] at (-1,1) {$G$};
 \node[vertex1_style, label=above:\small$v_1$] (z1) at  (.8,0.54) {};
 \node[vertex1_style, label=left:\small$v_2$] (z2) at  (0.1,0) {};
 \node[vertex1_style, label=above:\small$v_5$] (z5) at  (1.5,0) {};
\node[vertex1_style, label=below:\small$v_3$] (z3) at  (.4,-0.8) {};
  \node[vertex1_style, label=below:\small$v_4$] (z4) at  (1.2,-0.8) {};
  \draw[thick] (z1)--(z2)--(z3)--(z4)--(z5)--(z1)--(z3)--(z5)--(z2)--(z4)--(z1);
 \node[vertex1_style, label=above:\small$v_6$] (z6) at  (2.6,0) {};
 \draw[thick] (z6)--(z5);
\node[vertex1_style, label=below:\small$v_7$] (z7) at  (2.25,-1) {};
\node[vertex1_style, label=below:\small$v_8$] (z8) at  (2.95,-1) {};
\node[vertex1_style, label=right:\small$v_9$] (z9) at  (3.35,-.8) {};
\node[vertex1_style, label=right:\small$v_{10}$] (z10) at  (3.7,-.3) {};
 \draw[thick] (z6)--(z7)--(z8)--(z6);
 \draw[thick] (z6)--(z9)--(z10)--(z6);
\node[vertex1_style, label=south east:\small$v_{11}$] (z11) at  (3.4,.7) {};
\node[vertex1_style, label=left:\small$v_{12}$] (z12) at  (2.6,1.4) {};
\node[vertex1_style, label=above:\small$v_{14}$] (z14) at  (5,2.1) {};
\node[vertex1_style, label=above:\small$v_{15}$] (z15) at  (3.85,1.8) {};
\node[vertex1_style, label=south east:\small$v_{13}$] (z13) at  (4.65,1) {};
 \draw[thick] (z6)--(z11)--(z12);
 \draw[thick] (z11)--(z15)--(z14)--(z13)--(z11)--(z13);
 \draw[thick] (z13)--(z15);
 \draw[thick] (z11)--(z14);
\end{scope}

\end{tikzpicture} }
\end{center}
  \caption{ \label{forte}  \small A clique tree $G \in \mathcal C \mathcal T$ and its associated tree $T_G$}
  \end{figure}
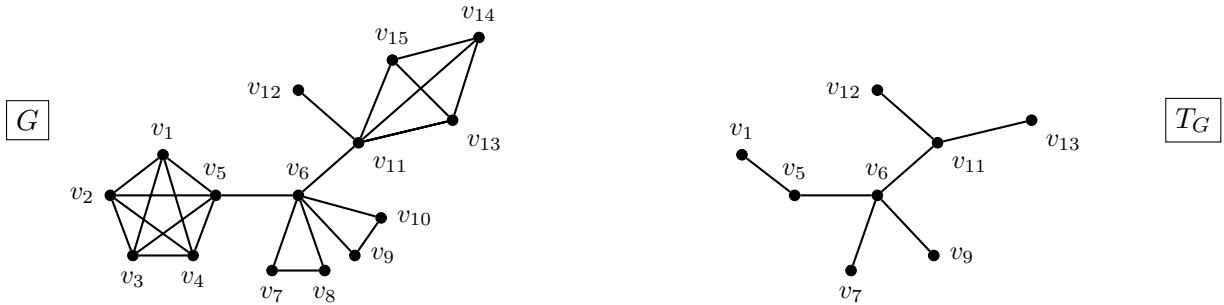

The following lemma, whose immediate proof only requires the very basic structural properties of the graphs involved, turns out to be useful in several proofs.

\begin{lem}\label{luno}  Let $G$ be a clique tree in $\mathcal C \mathcal T^{\geqslant 2}$. For each $u \in V_G$, there exists a vertex $v \in T_G$ such that
$e_G(u) = d_G(u,v)$.
\end{lem}

\begin{lem}\label{punoemezzo} Let $G$ be a clique tree in $\mathcal C \mathcal T^{\geqslant 2}$. For each $u \in V_{T_G} \subseteq V_G$ the eccentricities \textcolor{blue}{$e_G(u)$=$e_{T_G}(u)$}.
\end{lem}
\begin{proof}The inequality  $e_{T_G}(u) \leqslant e_G(u)$ holds since $T_G$ is an induced subgraph of $G$. The rest of the proof is devoted to prove the opposite inequality, i.e.\ $e_{T_G}(u) \geqslant e_G(u)$. Lemma \ref{luno} ensures that $e_G(u)=d_G(u,v)$  for a suitable vertex $v \in V_{T_G}$. In other words, there exists a path $P$ in $G$ connecting $u$ and $v$ such that $\lvert E_P\rvert=e_G(u)$. By definition, $P$ has the minimal lenght among the paths connecting $u$ and $v$, that is why the set $V_P \setminus \{u,v\}$  is included in $\mathcal C_G$ or, equivalently,  $P$ is also a subgraph of $T_G$. Therefore, $e_G(u)=d_G(u,v) = d_{T_G}(u,v) \leqslant e_{T_G}(u)$, and the proof is over.
\end{proof}

\begin{prop}\label{pdue} Let $G$ be a clique tree in $\mathcal C \mathcal T^{\geqslant 2}$. The matrix $\mathcal E(T_G)$ is the principal submatrix of $\mathcal E(G)$ obtained by considering rows and columns indexed in $V_{T_G}$.
\end{prop}
\begin{proof} The entries $\mathcal E(T_G)_{uv}$ and  $\mathcal E(G)_{uv}$ are equal for every pair $(u,v) \in V_{T_G} \times V_{T_{G}}$, since, by Lemma \ref{punoemezzo}, for every vertex $u$ in $V_{T_G}$, we have $e_{T_G}(u) = e_G(u)$.
\end{proof}

From Proposition \ref{pdue} it immediately comes the following useful corollary. In its statement and throughout the paper, we write  $\Gamma_{G}  $ and $\Gamma_{T_G}$ to  respectively denote the graphs $\Gamma(\mathcal E (G))$ and  $\Gamma(\mathcal E (T_G))$.
\begin{cor}\label{c1} For each clique tree $G$ in $\mathcal C \mathcal T^{\geqslant 2}$, The graph $\Gamma_{T_G}$ is an induced subgraph of $\Gamma_{G}$.
\end{cor}

\begin{prop}\label{2.5} For each clique tree $G$ in $\mathcal C \mathcal T$, the matrix $\mathcal E(G)$ is irreducible.
\end{prop}
\begin{proof} Let $G$ be a graph in  $\mathcal C \mathcal T$. If ${\rm diam} (G)=1$, then $G$ is complete and the matrix $\mathcal E (G)=A(G)$ is irreducible, being the adjacency matrix of a connected graph. If ${\rm diam}(G)=2$, then  $G$ is a clique star, i.e.\ the coalescence at a single cut-vertex of a bunch of complete graphs. In this case, the irriducibility of $\mathcal E(G)$ comes from \cite[Theorem~3.1]{AKP}.

Let now ${\rm diam }(G) \geqslant 3$.  By Theorem \ref{2.4} it will suffice to show that $\Gamma_G$ is connected.
 Theorem \ref{2.4}, together with Theorem \ref{2.3} and Proposition \ref{puno}, also  guarantees that $\mathcal E(T_G)$ is irreducible and the graph $\Gamma_{T_G}$  is connected.
The proof will end by showing that each pair $u_{1},u_{2}$ of distinct vertices of $G$ are connected by a suitable path in $\Gamma_G$.

If $\mathcal E(G)_{u_1u_2}= {\rm min} \{e_G(u_{1}), e_G(u_{2})\}\neq0$, then $u_1$ and $u_2$ are adjacent in $\Gamma_G$ and there is nothing else to prove.

 If $\mathcal E(G)_{u_1u_2}= 0$, by Lemma \ref{luno} we find in $V_{T_G}$ two vertices $v_1$ and $v_2$ such that $e_G(u_{i})=d(u_{i}, v_{i})$ for  $i \in \{1,2\}$. It follows that $\mathcal E(G)_{u_{i}v_{i}}= {\rm min} \{e_G(u_{i}),e_G(v_{i})\}\neq 0$, implying that $u_{i}$ and $v_{i}$ are adjacent in $\Gamma_G$. A path in $\Gamma_G$ connecting $u_1$ and $u_2$ is obtained by gluing the edge $u_1v_1$, a path between $v_1$ and $v_2$ (which exists since $\Gamma_{T_G}$ is connected and, by Corollary \ref{c1}, is an induced subgraph of $\Gamma_G$) and the edge $v_2u_2$.
\end{proof}

We end this section with a result on the eccentricities of noncut-vertices in a clique tree with at least two blocks.
\begin{prop}\label{ptre} Let $G$ be a clique tree in $\mathcal C \mathcal T^{\geqslant 2}$.\\
(i) If $u$ is a noncut-vertex of $G$ belonging to a leaf-block $B$, then $e_G(u)=e_G(w)+1$, where $w$ is the only cut-vertex of $B$.\\
(ii) If $u$ is a noncut-vertex of $G$ belonging to a bridge-block $B$, then $e_G(u)= \max \{e_G(w_1), e_G(w_2) \}$, where $w_1$ and $w_2$ are the two cut-vertices of $B$.
\end{prop}
\begin{proof} Let $v'$ be a vertex in $G$ such that $e_G(u)=d_G(u,v')$. Since $ {\rm diam} (G)\geqslant 2$, then $v'$ belongs to a block $B'\not=B$, and the minimal path $P$ connecting $u$ and $v'$ surely contains the cut-vertex $w$. Thus $e_G(w) \geqslant e_G(u)-1$. Part (i) will be proved once we show that
\begin{equation}\label{maueq1}
e_G(w) \leqslant e_G(u)-1
\end{equation} holds as well. In order to see this, it is clear that there exists a vertex $v''$ belonging to a block $B''\neq B$ such that $e_G(w) = d_G(w,v'')$. This means that $e_G(u) \geqslant d(u,v'')= d_G(w,v'')+1 = e_G(w)+1$, proving \eqref{maueq1}.

Now we prove Part (ii). It is not restrictive to assume $e_G(w_1) \geqslant e_G(w_2)$. We shall prove that for a vertex $u$ satisfying the hypothesis, we have
$e_G(u)=e_G(w_1)$. Let $z$ be a vertex of $G$ such that $e_G(w_1)=d_G(w_1,z)$. A minimal path \textcolor{blue}{P} connecting $w_1$ to $z$ surely contains $w_2$,
 otherwise $e_G(w_2) \geqslant d(w_2,z) =d(w_1,z)+1> e_G(w_1)$, a contradiction. It is obvious that a minimal path connecting $u$ and $z$ can be obtained from $P$ by replacing the edge \textcolor{blue}{$w_1w_2$} with \textcolor{blue}{$u w_2$}. Therefore $e_G(u) \geqslant d_G(u,z) =d_G(w_1,z)=e_G(w_1)$. To complete the proof, we need to show that the inequality
\begin{equation}\label{maueq2}
e_G(u)\leqslant e_G(w_1)
\end{equation} holds as well.
Let $v'$ be a vertex of $G$ such that $e_G(u)=d_G(u,v')$, and let $P'$ be a minimal path connecting $u$ to $v'$. Surely $v'$ belongs to a block $B'\neq B$.
This means that $\lvert V_{P'} \cap \{ w_1, w_2 \}\rvert =1$.
We now distinguished two cases. If $w_1$ is in  $V_{P'} $, then $d_G(w_2,v')=d_G(u,v')$; therefore $e_G(w_1) \geqslant e_G(w_2) \geqslant
d_G(w_2,v') = e_G(u)$ and we are done. If, otherwise, $w_2$ belongs to $V_{P'}$, then $d_G(w_1,v')=d_G(u,v')$, therefore $e_G(u) = d_G(w_1,v') \leqslant e_G(w_1)$. Hence \eqref{maueq2} holds in this case too.
\end{proof}

\section{Inertia of eccentricity matrices of clique trees in $\mathcal C \mathcal T$}
 As already noted by Jordan in the nineteenth century,  every tree with even diameter has exactly one center; whereas a tree with odd diameter has two centers, necessarily adjacent   (a proof of this fact can be found in \cite[Theorem 4.2]{Har}).
The next theorem generalizes this classical result to cliques tree in $\mathcal C \mathcal T$. In order to make more concise its statement, we set
\[ \mathcal C \mathcal T^{\geqslant 2}_{ev} = \{ G \in \mathcal C \mathcal T^{\geqslant 2} \; \mid \text{ ${\rm diam} (G)$ is even} \} \quad \text{and}
\quad \mathcal C \mathcal T^{\geqslant 2}_{odd} = \{ G \in \mathcal C \mathcal T^{\geqslant 2} \; \mid \text{ ${\rm diam} (G)$ is odd} \}.   \]

\begin{thm}\label{3.3} Let $G$ be  a clique tree in $\mathcal C \mathcal T$, and let $C(G)$ denote its center.\\[.2em]
(i) if $G$ is complete, then $C(G)=V_G$;\\[.2em]
(ii) if $G$ belongs to  $\mathcal C \mathcal T^{\geqslant 2}_{ev}$,  then $G$ has exactly one central vertex; namely, the center of $T_G$;\\[.2em]
(iii) if $G$ belongs to  $\mathcal C \mathcal T^{\geqslant 2}_{odd}$,  then the center $C(G)$ is the block of $G$ containing the two centers on $T_G$.
\end{thm}
\begin{proof} Part (i) is trivially true. We now deal with Part (ii) and Part (iii) simultaneously. Let $G$ be a clique tree in $ \mathcal C \mathcal T^{\geqslant 2}$.
We start by claiming that ${\rm rad} (G)$
is attained by at least one vertex in $T_G$. This fact immediately comes from the two parts of Proposition~\ref{ptre}. Lemma~\ref{punoemezzo} ensures that
$e_G(v) = e_{T_G}(v)$ for every $v\in V_{T_G}$. This means that  $\mathcal C_G \cap C(G)$, i.e.\ the set of centers of $G$ which are also cut-vertices, is equal to $C(T_G)$. As already recalled, $C(T_G)$ is a singleton, say $\{z\}$, if ${\rm diam} (G)= {\rm diam} (T_G)$ is even, and consists of two adjacent vertices in $T_G$, say $z_1$ and $z_2$,  if $ {\rm diam} (G)= {\rm diam} (T_G)$ is odd (and larger than $1$).  In the latter case, by definition of $T_G$, the vertices $z_1$ and $z_2$ are the two cut-vertices of a special block of $G$, say $B_C$. The proof ends by showing that
\[ C(G) \cap (G \setminus \mathcal C_G) = \begin{cases} \varnothing \qquad \qquad \qquad \text{if $G \in\mathcal C \mathcal T^{\geqslant 2}_{ev}$},\\[.3em]
                                                                                   B_C \setminus \{z_1,z_2\}\quad  \, \text{if $G \in\mathcal C \mathcal T^{\geqslant 2}_{odd}$}.\\
\end{cases}
\]
Noncut-vertices of a leaf-block surely are not in $C(G)$ by Proposition~\ref{ptre}(i). Consider now a noncut-vertex $v$ belonging to a bridge-block $B$, and suppose $V_B \cap V_{T_G}=\{w_1, w_2\}$. By Proposition~\ref{ptre}(ii) we deduce that $v$ belongs to $C(G)$ if and only if
\[ e_G(v)= e_{G} (w_1)=e_{G} (w_2) \qquad    \text{and} \qquad \{w_1,w_2\} \subseteq \mathcal C_G \cap C(G) = C(T_G),\]
and this happens if and only if the diameter is odd and $B=B_C$.
\end{proof}
Let $G$ be a clique tree with odd diameter. Theorem \ref{3.3} makes reasonable to refer to the block of its centers as the {\em central block}.

The next theorem shows that,  for $G \in \mathcal C \mathcal T^{\geqslant 2}_{odd}$,  the rank of $\mathcal E(G)$ and the number of eigenvalues sharing the same sign are not affected by the order of the single blocks and how these are mutually disposed. Our result generalizes to clique trees in $\mathcal C \mathcal T^{\geqslant 2}_{odd}$  what Theorem 3.1 in \cite{MK} states for trees.

Throughout the rest of the paper, for $U\subseteq V_G$, we denote by $\mathcal E(G)_U$ the $\lvert U \rvert \times \lvert V_G \rvert$ submatrix of $\mathcal E(G)$ consisting of the rows indexed by the vertices in $U$. Moreover, we set $ \mathcal E(G)_u := \mathcal \mathcal E(G)_{\{u\}}$. With a slight abuse of notation we denote all the identity matrices, all the null matrices and all-ones matrices by $I$,  $O$ and $J$ respectively.  \textcolor{blue}{Their sizes} will be clear from the context in any of their occurrences.
\begin{thm}\label{3.5} Let $G$ be a clique tree of order $n$ in  $\mathcal C \mathcal T^{\geqslant 2}_{odd}$. Then, ${\rm rk} (\mathcal E(G))=4$. Moreover, $\mathcal E(G)$ has exactly two positive and two negative eigenvalues, that is, ${\rm In} (\mathcal E(G))=(2, 2, n-4)$.
\end{thm}
\begin{proof} Since $G$ belongs to $\mathcal C \mathcal T^{\geqslant 2}_{odd}$, then ${\rm diam} (G) = 2k+1$ for a certain $k\in \N$; moreover, $n= \lvert V_G \rvert \geqslant 4$ and, by Proposition \ref{puno}, ${\rm diam} (T_G) = 2k+1$. Theorem \ref{3.3} ensures that $C(G)=B_C$, the block of $G$ including $C_{T_G}=\{z_1,z_2 \}$.  The graph $G'$ obtained from $G$ by removing all the edges of the central block $B_C$ has two connected components. We denote by $T_1$ (resp. $T_2$) the component of $G'$ containing $z_1$ (resp. $z_2$).

We set \\
\[ \begin{array}{lllll}
 V_{1}:=\{u\in V_{T_{1}}  \mid  d(u,z_{1})=k\}, & V_{2}:=\{v\in V_{T_{2}} \mid d(v,z_{2})=k\},\\[.4em]
 V_{3}:=\{u\in V_{T_{1}} \mid 0\leqslant d(u,z_{1})<k\}, &
 V_{4}:=\{v \in V_{T_{2}} \mid 0\leqslant d(v,z_{2})<k\},\ \textcolor{blue}{V_{5}:=V_{B_C} \setminus \{z_1,z_2\}}.
\end{array}
\]
 Clearly, $\{ V_i \mid 1 \leqslant i \leqslant 5\}$ is a partition
 of $V_G$. We just need the definition of $\mathcal E(G)$ and of the several $V_i$'s to realize that
\[ \begin{array}{lllllll}
\mathcal E(G)_{uv}=0, &  \text{if $\{u,v\} \subseteq V_{i}$ ($1 \leqslant i \leqslant 5$),} && \mathcal E(G)_{uv}=e(u) & \text{if $(u,v) \in V_{4} \times V_{1}$}, \\
\mathcal E(G)_{uv}=2k+1, & \text{if $(u,v) \in V_{1} \times V_{2}$}, && \mathcal E(G)_{uv}=k+1, & \text{if  $(u,v) \in V_{5} \times (V_{1} \cup V_{2})$},\\
\mathcal E(G)_{uv}=0, &  \text{if  $(u,v) \in (V_{1} \times V_{3}) \cup (V_{2} \times V_{4}) $}, &&
\mathcal E(G)_{uv}=0, &   \text{ if $(u,v) \in V_{5} \times (V_{3} \cup V_{4})$},\\
\mathcal E(G)_{uv}=e(v) & \text{if  $(u,v) \in V_{2} \times V_{3}$}.
\end{array}
\]
Consequently, the matrix $\mathcal E(G)$, which is  symmetric, can be partitioned in the following way:
\begin{equation}\label{Eq:matrix1}
\mathcal E(G)=\bordermatrix{%
       & V_1       & V_2          &V_3     &V_4          &V_5\cr
V_1    & O        &(2k+1)J       &O       &P            &(k+1)J\cr
V_2    &(2k+1)J    &   O          &Q       &O           &(k+1)J\cr
V_3    & O         &Q^{\top}            &O       &O            &O\cr
V_4    & P^{\top}        & O            &O       &O            &O\cr
V_5    & (k+1)J    & (k+1)J       &O       &O            &O\cr
},
\end{equation}
where $P$ (resp. $Q$) is a suitable $\lvert V_1 \rvert \times  \lvert V_4 \rvert $ (resp.  $\lvert V_2 \rvert \times  \lvert V_3 \rvert$) matrix. Note that all rows of $P$ (resp. $Q$) are identical and nowhere zero. As a consequence, the rows of $\mathcal E (G)$ corresponding to vertices in $V_1$ (resp. $V_2$) are all equal, and those corresponding to vertices in $V_3$ (resp. $V_4$) are proportional; thus, ${\rm rk} (\mathcal E(G)_{V_{i}})=1$ for $1 \leqslant i \leqslant 4$.  Note that each row  of the fifth block  in $\mathcal E(G)$ can be written as
\[ (k+1) \left( \frac{1}{e_G(u')} \mathcal E(G)_{u'} +  \frac{1}{e_G(u'')} \mathcal E(G)_{u''} \right), \]
where $u' $ is any vertex of $V_3$ and $u''$ any vertex of $V_4$. Thus, $ {\rm rk}( \mathcal E(G))\leqslant 4$. Actually, the rank of $\mathcal E(G)$ is $4$ \textcolor{blue}{since there exist a invertible matrix $A$ which is the principal submatrix of $\mathcal E(G)$}
\begin{equation*}
A=\ \bordermatrix{%
       & u_1       & u_2          &w     &z         \cr
u_1    & 0         &2k+1          &0       &k+1         \cr
u_2    &2k+1       & 0            &k+1     &0           \cr
w    & 0         &k+1           &0       &0           \cr
z    &k+1        & 0            &0       &0           \cr
},
\end{equation*}
 indexed by $\{u_1, u_2, w, z\}$, where  $u_{i}\in V_{i}$ for $i=1, 2$, $w\in V_{3}$ and  $z \in V_{4}$.

Now, a direct computation shows that ${\rm In}(A)=(2,2,0)$. More precisely,
\[{\rm Spec} (A) =\left\{ \frac{q+2k+1}{2},  \frac{q-2k-1}{2}, -\frac{q-2k-1}{2}, - \frac{q+2k+1}{2}  \right\}
, \qquad  \text{where $q= \sqrt{5+12k+8k^2}.$} \]  \textcolor{blue}{By Theorem 2.2}, we deduce that $\mathcal E(G)$ has at least two positive and two negative eigenvalues. But $ {\rm rk}(\mathcal E(G))=4$, hence the positive (resp. negative) eigenvalues of $\mathcal E(G)$ are exactly  two. In other words, $In(\mathcal E(G))=(2, 2, n-4)$.
\end{proof}
The only graphs with odd diameter in $\mathcal C \mathcal T$ left out of Theorem \ref{3.5} are the complete graphs of order $n \geqslant 2$, and it is well-known that ${\rm In}(\mathcal E(K_n))=(1,n-1,0)$ for $n \geqslant 2$, since $\mathcal E(K_n)$ is simply the adjacency matrix of $K_n$.

Before computing ${\rm In}(\mathcal E(G))$ for $ G \in \mathcal C \mathcal T^{\geqslant 2}_{ev}$ we need a definition and two lemmas.

\begin{Def}\label{3.1} Let $G$ be a clique tree in $\mathcal C \mathcal T^{\geqslant 2}_{ev}$, and let $z$ be the only center of $G$. A vertex $v$ adjacent to $z$ in $G$ is said to be `diametrally distinguished' if there is a diametral path containing the vertex $v$. The set of diametrally distinguished vertices will be denoted by $\mathcal D_G$.
\end{Def}

\begin{lem}\label{ldue} Let $G$ be a clique tree in $\mathcal C \mathcal T^{\geqslant 2}_{ev}$ with ${\rm diam}(G) \geqslant 4$. Then, $\mathcal D_G \subset \mathcal C_G$, i.e. the diametrally distinguished vertices of $G$ are also cut-vertices.
\end{lem}
\begin{proof} It is sufficient to note, for every $G \in \mathcal C \mathcal T^{\geqslant 2}$ the non-pendant vertices of a diametral path are all cut-vertices.
\end{proof}

\begin{lem}\label{3.1}{\rm \cite[Lemma 3.1]{MK}} Let $A$ be a $2n\times2n$ symmetric matrix partitioned as
$$A= \begin{pmatrix} 2k(J-I)&(2k-1)(J-I) \\ (2k-1)(J-I)&O  \end{pmatrix}, \quad \text{ for $k \in \N$.}$$
Then, the inertia of $A$ is $(n, n, 0)$.
\end{lem}

\begin{thm}\label{3.6} Let $G$ be a clique tree in $\mathcal C \mathcal T^{\geqslant 2}_{ev}$ with order $n$ and diameter $d=2k$ with $k \geqslant 2$, and let $C(G)=C(T_G)=\{z\}$.  Then, ${\rm rk} (\mathcal E(G))=2l$ and ${\rm In} (\mathcal E(G))=(l, l, n-2l),$ where $l:= \lvert \mathcal D_G\rvert$.
\end{thm}
\begin{proof}
Let $N_G(z) \cap \, \mathcal C_G = N_{T_G}(z) = \{ w_1, \dots, w_p\}$. \textcolor{blue}{Without loss of generality to assume that} $N_{T_G}(z) \cap \mathcal D_G= \{ w_i \mid 1 \leqslant i \leqslant l \}$. For $1 \leqslant i \leqslant l$, we denote by $B_i$ the bridge-block containing $z$ and $w_i$, and by $H_i$
the graph obtained from $G$ by removing all the edges of $B_i$. Let $T_i$ be the component
of $H_i$ containing $w_i$, and let $T_{l+1} := G \setminus \left( \bigcup_{i=1}^l T_i \right)$.
 It is useful to note that $e_{T_i}(w_i)=k-1$ for $1 \leqslant i \leqslant l$. The several $T_i$'s allow to obtain a partition \textcolor{blue}{$\mathcal V =\{ V_j \mid 1 \leqslant j \leqslant 2l+1 \}$ }of the set $V_G$
in the following way:

\[ \begin{array}{ @{}  l @{} }
      \;\;\; V_{i}=\{u\in T_{i} \mid d(u,z)=k\} \\[1.5\jot]
      V_{l+i}=\{u\in T_{i} \mid 0\leqslant d(u,w_{i})< k-1\}
    \end{array}
 \qquad \text{(for $1 \leqslant i \leqslant l$),} \quad \text{and} \quad V_{2l+1}=V_{T_{l+1}}
\]
(if needed, the reader can consult Example \ref{ex3} in Section 4).
Taking into account that the middle point of any diametral path is $G$ is necessarily $z$, and the fact that a vertex $  v$ belongs to $\bigcup_{i=1}^l V_i$ if and only if $v$ is both a noncut-vertex of a leaf-block and an endpoint on a diametral path, the following equalities, holding for $\{i,j\} \subseteq \{1,2,\dots, l\}$, are easy to check:
\[\label{myeq1} \begin{array}{lllll}
\mathcal E(G)_{uv}=0 &   \text{for $\{u,v\} \subseteq V_{i}$,} \\
\mathcal E(G)_{uv}=0&    \text{for $\{u,v\} \subseteq V_{l+i}$,}  && \mathcal E(G)_{uv}=2k &   \text{for $(u,v) \in V_{i} \times V_{j} \quad \;\,$ and $i\neq j$,}\\
\mathcal E(G)_{uv}=0 &   \text{for $\{u,v\} \subseteq V_{2l+1}$, } &&\mathcal E(G)_{uv}=e(v)   &\text{for $(u,v) \in V_{i} \times V_{l+j} \;\;\,$ and $i\neq j$,}\\
\mathcal E(G)_{uv}=0 & \text{for $(u,v) \in V_{i} \times V_{l+i}$,}&& \mathcal E(G)_{uv}=0 &\text{for $(u,v) \in V_{l+i} \times V_{l+j}$ and $i\neq j$,}\\
\mathcal E(G)_{uv}=0    & \text{for $(u,v) \in V_{l+i} \times V_{2l+1}$,}\\
\mathcal E(G)_{uv}=e(v) &  \text{for $(u,v) \in V_{i} \times V_{2l+1}$}\
\end{array}
\]
(each of the nine equalities above requires its own elementary argument). Therefore, the eccentricity matrix $\mathcal E(G)$ can be accordingly partitioned  as
\begin{equation}\label{Eq:matrixk}
\mathcal E(G)=\ \bbordermatrix{%
       & V_1       & V_2          &\cdots     &V_l          &V_{l+1}         &V_{l+2}        &\cdots          &V_{2l}         &V_{2l+1}\cr
V_1    & O         &(2k)J         &\cdots     &(2k)J        &O              &P_{1,l+2}      &\cdots          &P_{1,2l}       &P_{1,2l+1}\cr
V_2    &(2k)J      & O            &\cdots     &(2k)J        &P_{2,l+1}       &O             &\cdots          &P_{2,2l}       &P_{2,2l+1}\cr
\;\, \vdots &\vdots     &\vdots        &\ddots     &\vdots       &\vdots          &\vdots         &\ddots          &\vdots         &\vdots\cr
V_l    &(2k)J      &(2k)J         &\cdots     &O            &P_{l,l+1}       &P_{1,l+2}      &\cdots          &O              &P_{l,2l+1}\cr
V_{l+1}  & O         &P_{2,l+1}^{\top}  &\cdots     &P_{l,l+1}^{\top}   &O               &O              &\cdots          &O             &O\cr
V_{l+2}  &P_{1,l+2}^{\top} &O           &\cdots     &P_{l,l+2}^{\top}   &O               &O              &\cdots          &O             &O\cr
\;\, \vdots &\vdots     &\vdots        &\ddots     &\vdots       &\vdots          &\vdots         &\ddots          &\vdots         &\vdots\cr
V_{2l}   &P_{1,2l}^{\top}  &P_{2,2l}^{\top}  &\cdots     &O           &O               &O             &\cdots          &O              &O\cr
V_{2l+1} &P_{1,2l+1}^{\top}&P_{2,2l+l}^{\top} &\cdots     &P_{l,2l+1}^{\top}  &O             &O            &\cdots          &O            &O\cr
},
\end{equation}
where the rows of each matrix $P_{i,l+j}$ with $i\neq j$ and $(i,j) \in \mathcal S:=\{1, 2,\dots , l \} \times \{1, 2,\dots, l, l + 1\}$  are identical and nowhere zero. Moreover, all the nonzero entries of the column corresponding to a vertex $v \in V_{l+j}$ are the same, clearly implying that,
for a fixed $(i,j) \in \mathcal S$, the columns of the matrices $P_{i,l+j}$ are multiples of an all-ones-vector of appropriate size.

For $1 \leqslant i \leqslant l$, the rows corresponding to vertices in a fixed $V_{i}$ (resp. $V_{l+ i}$) are all the same (resp. all proportional); thus, ${\rm rk} (\mathcal E(G)_{V_{i}})={\rm rk} (\mathcal E(G)_{V_{l+i}})= 1$.

Since for each  $v\in V_{2l+1}$ we have
\[  \mathcal E(G)_v =   \sum_{i=1}^l  \left(\frac{e_G(v)}{(l-1)e_G(w_i)}  \mathcal E(G)_{w_i} \right) \]
with each $w_i$ arbitrarily chosen in $V_{l+i}$,
we have
 ${\rm rk} (\mathcal E(G))\leqslant 2l$. In order to see that ${\rm rk} (\mathcal E(G))= 2l$, \textcolor{blue}{we show a invertible matrix $M$ which is the principal submatrix of $\mathcal E(G)$ }indexed by the vertices $v_{1}, v_{2},\dots, v_{l}, z_{1},\dots, z_{l}$, where, for $1 \leqslant i \leqslant l$, $v_i$ is any vertex in $V_i$, whereas $z_i$ is the only cut-vertex of the leaf-block containing  $v_i$. By definition,
$z_i$ belongs to $V_{l+i}$ and $e_G(z_i)=2k-1$. That is why the matrix $M$ assumes the following form:
\begin{equation*}
M=\ \bordermatrix{%
       & v_1       & v_2          &\cdots     &v_l          &z_{1}           &z_{2}          &\cdots          &z_{l}           \cr
v_1    & 0         &2k            &\cdots     &2k           &0               &2k-1           &\cdots          &2k-1           \cr
v_2    &2k         & 0            &\cdots     &2k           &2k-1            &0              &\cdots          &2k-1           \cr
\,\;\vdots &\vdots &\vdots        &\ddots     &\vdots       &\vdots          &\vdots         &\ddots          &\vdots         \cr
v_l    &2k         &2k            &\cdots     &0            &2k-1            &2k-1           &\cdots          &0              \cr
z_{1}& 0           &2k-1          &\cdots     &2k-1         &0               &0              &\cdots          &0              \cr
z_{2}&2k-1         &0             &\cdots     &2k-1         &0               &0              &\cdots          &0              \cr
\,\;\vdots &\vdots     &\vdots        &\ddots     &\vdots       &\vdots          &\vdots         &\ddots          &\vdots         \cr
z_{l} &2k-1       &2k-1          &\cdots     &0            &0               &0              &\cdots          &0              \cr
},
\end{equation*}
Clearly, $M$ can be partitioned as
$$M= \begin{pmatrix} 2k(J-I)&(2k-1)(J-I) \\ (2k-1)(J-I)&O  \end{pmatrix},$$ and ${\rm rk}\,(M)=2l$ by Lemma \ref{3.1}.

The same lemma also ensures that ${\rm In} (M)=(l, l, 0)$. \textcolor{blue}{By Theorem 2.2}, $\mathcal E(G)$ has at least $l$ positive and $l$ negative eigenvalues. Since we know that  ${\rm rk}(\mathcal E(G)))=2l$, the matrix $\mathcal E(G)$ has exactly $l$ positive and $l$ negative eigenvalues or, equivalently, ${\rm In}(\mathcal E(G))=(l, l, n-2l)$ as claimed.
\end{proof}
\begin{re} {\rm The hypothesis $k\geqslant 2$ in the statement of Theorem \ref{3.6} cannot be removed. If ${\rm diam}(G)=2$ for a clique tree $G \in \mathcal C \mathcal T^{\geqslant 2}_{ev}$,  then $G$ is a clique star with order $n \geq 3$, and $l:=\lvert \mathcal D_G \rvert =n-1$. In this case, Theorem 3.4 in \cite{AKP} says that ${\rm rk}(\mathcal E(G))= t+1$ and ${\rm In} (\mathcal E(G))=(1,n-t-1 , t)$, where $t\geqslant 2$ is the number of the blocks of $G$. Besides, the $3$-tuple $(l,l,n-2l)$ is not nonnegative; hence, it cannot be the inertia  of any matrix.}
\end{re}

\section{$\mathcal E$-spectral syimmetry and clique trees in $\mathcal C \mathcal T$}
In this section, we detect the clique trees in $\mathcal C \mathcal T$ whose $\mathcal E$-eigenvalues are symmetric with respect to $0$. If this is the case, we simply say that the $\mathcal E$-spectrum is {\em symmetric}. We first focus on graphs with odd diameter.

\textcolor{blue}{
\begin{lem}\label{4.1}{\rm \cite[Lemma 4.2]{MK}} Let $p(\lambda) = \lambda^{n}+c_{1}\lambda^{n-1}+. . . +c_{n-1}\lambda+ c_{n}$ be a polynomial such that all of its roots are non-zero and real numbers. If $c_{i}$ and $c_{i+1}$ are different from zero for some $i\in\{1,\cdots, n-1\}$, then the roots of $p(\lambda)$ are not symmetric about the origin, that is, there exists a real number $\lambda_{0}$ such that $p(\lambda_{0})=0$ and $p(-\lambda_{0})\neq0$.
\end{lem}
}
\begin{thm}\label{4.1} Let $G$ be a clique tree in $\mathcal C \mathcal T$ with odd diameter. The $\mathcal E$-spectrum of $G$ is symmetric if and only if $\lvert C(G) \rvert =2$.
\end{thm}

\begin{proof} In the set of complete graphs with odd diameter, i.e. $\{K_n \mid n\geqslant 2\}$, the only graph whose $\mathcal E$-spectrum is symmetric is $K_2$. In fact, ${\rm Spec}_{\mathcal E}(K_2)=\{1,-1\}$ and ${\rm In}(\mathcal E(K_n))=(1,n-1,0)$ for all $n \geqslant 2$. This is consistent with our statement since $\lvert C(K_n) \rvert =n$.

Now, let $G$ be a fixed clique tree in $\mathcal C \mathcal T_{odd}^{\geqslant 2}$. By Theorem \ref{3.3}(iii) we know that $C(G)=B_C$, the block containing the two adjacent centers $z_1$ and $z_2$ of $T_G$. As in the proof of Theorem \ref{3.5}, we consider the graph $G'$ obtained from $G$ by removing all edges of the central block $B_C$, and denote by $T_1$ (resp. $T_2$) the component of $G'$ containing $z_1$ (resp. $z_2$). Four our purposes, it is convenient to shuffle the $V_i$'s  considered in the proof of Theorem \ref{3.5}. More explicitly, we set
\[ \begin{array}{lllll}
 W_1 := V_{1}=\{u\in V_{T_{1}}  \mid  d(u,z_{1})=k\}, & W_3:= V_{2}=\{v\in V_{T_{2}} \mid d(v,z_{2})=k\},\\[.4em]
 W_2:= V_{3}=\{u\in V_{T_{1}} \mid 0\leqslant d(u,z_{1})<k\}, &
 W_4:= V_{4}=\{v \in V_{T_{2}} \mid 0\leqslant d(v,z_{2})<k\},
\end{array}
\]
 and
 $W_5:=V_{5}=V_{B_C} \setminus \{z_1,z_2\}$.

Suppose now  $\lvert B_C \rvert =2$. In this case, $W_5$ is empty. Therefore,
the matrix $\mathcal E(G)$  can be written as\\
\begin{equation}\label{Eq:matrix10}
\mathcal E(G)=\ \bordermatrix{%
       & W_1       & W_2          &W_3           &W_4          \cr
W_1    & O         &O             &(2k+1)J       &P            \cr
W_2    & O        &O             &Q^{\top}             &O            \cr
W_3    &(2k+1)J    &Q            &O            &O           \cr
W_4    &P^{\top}         &O             &O             &O            \cr
},
\end{equation}
where the blocks $P$ \textcolor{blue}{and} $Q$ are precisely those appearing in \eqref{Eq:matrix1}. \textcolor{blue}{Partitioning the eccentricity matrix of $G$ with respect to the vertices of $W_{1}\bigcup W_{2}$ and $W_{3}\bigcup W_{4}$ gives us}\\
\[ \mathcal E(G)= \begin{pmatrix} O &  C\\ C^{\top} &O  \end{pmatrix},  \qquad \text{where} \quad C: = \begin{pmatrix} (2k+1)J & P \\ Q^{\top} & O \end{pmatrix}.
\]
Let now ${\bf x}$ be a $\lambda$-eigenvector of  $\mathcal E(G)$. If we write ${\bf x} = {{\bf y}\choose{\bf z}}$, where $\bf y$ consists of the first $\lvert W_1 \cup W_2\rvert$ components of ${\bf x}$, it is straightforward to check that $-\lambda$ also belongs to ${\rm Spec}_{\mathcal E}(G)$, the vector  ${{\;\bf y}\choose{-\bf z}}$ being one of its eigenvectors. The same argument shows that $\lambda$ and $-\lambda$ have the same multiplicity. Thus, ${\rm Spec}_{\mathcal E}(G)$ is symmetric, as claimed.

We now assume that $W_5$ is nonempty. Our goal is to show that ${\rm Spec}_{\mathcal E}(G)$ is not symmetric. This time,
$\mathcal E(G)$  can be written as
\begin{equation}\label{Eq:matrix11}
\mathcal E(G)=\ \bordermatrix{%
       & W_1       & W_2          &W_3           &W_4    &W_5      \cr
W_1    & O         &O             &(2k+1)J       &P        & (k+1)J    \cr
W_2    & O        &O             &Q^{\top}             &O   & O         \cr
W_3    &(2k+1)J    &Q            &O            &O     & (k+1)J      \cr
W_4    &P^{\top}         &O             &O             &O         & O   \cr
W_5 & (k+1)J & O & (k+1)J & O & O
}.
\end{equation}
By Theorem \ref{3.5}, The rank of $\mathcal E(G)$ is four (by Theorem \ref{3.5}) and its trace is $0$. By Theorem \ref{2.2}, the characteristic polynomial of $\mathcal E(G)$ assumes the form
\[ \det (\lambda I - \mathcal E(G)) = \lambda^{n-4} (\lambda^4 +a \lambda^2+b \lambda +c),\]
with $c\not=0$ and $b$ is the opposite of the sum of all the principal minors of size $3$.
We now prove that the $b$ is nonzero \textcolor{blue}{by Lemma 4.1}, implying that ${\rm Spec}_{\mathcal E}(G)$ is not symmetric.
By looking at \eqref{Eq:matrix11}, it is not hard to check
that the nonzero principal minors of size $3$ inside $\mathcal E(G)$ necessarily involve a first row indexed in $W_1$, a second row indexed in $W_3$, and a third row indexed in $W_5$. These nonzero minors are all equal to
\[ \det \begin{pmatrix} 0 & 2k+1 & k+1 \\ 2k+1 & 0 & k+1\\ k+1 & k+1 & 0 \end{pmatrix} = 2(2k+1)(k+1)^2. \]
That is why
\[ b= \textcolor{blue}{- \lvert W_1 \rvert \cdot  \lvert W_3 \rvert \cdot \lvert W_5}   \rvert \cdot 2(2k+1)(k+1)^2 \not=0\]
as wanted.
\end{proof}
\begin{cor} If $G$ is a clique tree in $\mathcal C \mathcal T^{\geqslant 2}_{odd}$ with $\lvert B_C\rvert=2$ and at least  $5$ vertices, then $G$ has exactly five pairwise distinct $\mathcal E$-eigenvalues.
\end{cor}
\begin{proof}   By Theorems \ref{3.5} and \ref{4.1}, $\mathcal E(G)$ has two positive and two negative eigenvalues, all appearing in the sequence of inequalities
\begin{equation}\label{kaz} \xi_1 \geqslant \xi_2 >0 > -\xi_2 \geqslant -\xi_1. \end{equation}
The number $0$ is also an $\mathcal E$-eigenvalues since $n-4>0$.
 Now, the first inequality (and consequently the last one) in \eqref{kaz} is strict, since by Proposition \ref{2.5}
$\mathcal E(G)$ is a nonnegative irreducible matrix, and the Perron-Frobenius theorem ensures that $\xi_{1}$ is simple. Therefore, the five eigenvalues occuring in \eqref{kaz} are pairwise distinct. This completes the proof.
\end{proof}

 Let $A$ be an $n\times n$ real symmetric matrix. Along the proof of Theorem \ref{4.1} we already used the following fact: if $A\not=O$, the spectrum of $A$ is symmetric with respect to the origin if and only if the characteristic polynomial of $A$ is of type $\lambda^{n-2t}(\lambda^2-a_1)\, \cdots \, (\lambda^2-a_t)$ for suitable nonzero real numbers $a_1,\dots, a_t$.

Let $G$ be a clique tree in $\mathcal C \mathcal T^{\geqslant 2}_{ev}$ with $n$ vertices. In the next \textcolor{blue}{theorem}, we show that the coefficients of (at least) two monomials of consecutive degrees in the characteristic polynomial of $\mathcal E(G)$ are nonzero. This implies that the $\mathcal E$-spectrum of $G$ is not symmetric.
\begin{thm}\label{4.5}  Let $G \in \mathcal C \mathcal T^{\geqslant 2}_{ev}$. The $\mathcal E$-spectrum of $G$ is not symmetric.
\end{thm}
\begin{proof} If ${\rm diam}(G)=2$, the matrix $\mathcal E(G)$ has just one positive eigenvalue and at least two negative eigenvalues  (see \cite[Theorem 3.4]{AKP}); thus ${\rm Spec}_{\mathcal E}(G)$ cannot be symmetric. From now on, we assume ${\rm diam} (G) \geqslant 4$. Let $C(G)=C(T_G)=\{z\}$.

Consider now the partition $\mathcal V$ of $V_G$ defined along the proof of Theorem \ref{3.6}. The matrix $\mathcal E (G)$ assumes the form \eqref{Eq:matrixk}, and its characteristic polynomial can be written as
$$\phi_{\mathcal E}(\lambda) = \lambda^{n-2l}(\lambda^{2l} + c_{1}\lambda^{2l-1} + \cdots + c_{2l-1}\lambda + c_{2l}),$$
with $c_{2l}\neq0$, since ${\rm rk} \, (\mathcal E(G))=2l$ by Theorem \ref{3.6}, and $c_1=0$ since the trace of $\mathcal E (G)$ is null.

In order to prove that the $\mathcal E$-spectrum of $G$ is not symmetric \textcolor{blue}{by Lemma 4.1}, we now show that $c_{2}$ and $c_{3}$ are both nonzero. As recalled in Theorem \ref{2.2}, the absolute value of
$c_{2}$ (resp. $c_3$) is equal to the absolute value of the sum $E_2$ (resp. $E_3$) of all principal minors of $\mathcal E(G)$ of size $2$ (resp. size $3$).
Now, since
\[ \det \begin{pmatrix} 0 & a \\ a & 0 \end{pmatrix} = -a^2 \qquad \text{and} \qquad  \det \begin{pmatrix} 0 & a & b \\ a & 0 & c \\ b & c & 0 \end{pmatrix} =2abc,\]
it follows that each principal minor of $\mathcal E(G)$ of size $2$ is nonpositive, and each principal minor of $\mathcal E(G)$ of size $3$ is nonnegative.
Consequently, $\lvert c_2 \rvert >0$ and $\lvert c_3 \rvert >0$ if and only if we find at least one nonzero principal minor of size $2$ and one nonzero principal minor of size $3$.
For $ i \in \{1,2\}$, let $v_i$ be a vertex in $V_i$. Since the principal submatrices
\[
A=\ \bordermatrix{%
       &v_1        & v_2              \cr
v_1    & 0         &2k            \cr
v_2    &2k         & 0            \cr
}, \qquad \text{and} \qquad
A'=\ \bordermatrix{%
       &v_1        & v_2          &z\cr
v_1    & 0         &2k            &k\cr
v_2    &2k         & 0            &k\cr
z    &k          &k             &0\cr
},
\]
have both a nonzero determinant, the proof is complete.
\end{proof}

Combining Theorems \ref{4.1} and \ref{4.5}, we realize that the only clique trees in $\mathcal C \mathcal T$ whose $\mathcal E$-spectrum is symmetric are those having exactly two central vertices.
\textcolor{blue}{
\begin{thm}\label{4.6} A clique tree $G$ in $\mathcal C \mathcal T$ is symmetric with respect to the origin if and only if $G$ has an odd diameter and exactly two adjacent central vertices.
\end{thm}
}

\section{Examples}
\begin{figure}
\begin{center}
{\begin{tikzpicture}[vertex1_style/.style={circle,draw,minimum size=0.14 cm,inner sep=0pt, fill=black},vertex2_style/.style={circle,draw,minimum size=0.07 cm,inner sep=0pt, fill=black}, nonterminal/.style={
rectangle,
minimum size=2mm,
thin,
draw=black,
top color=white, 
bottom color=white!50!white!50, 
font=\itshape
}]
   \begin{scope}[xshift=8em]
         \node[nonterminal] at (6.5,1) {$G_2$};

 \node[vertex1_style, label=above:\small$v_1$] (z5) at  (1.5,0) {};
 \node[vertex1_style, label=north west:\small$v_2$] (z6) at  (2.6,0) {};
 \draw[thick] (z6)--(z5);
\node[vertex1_style, label=below:\small$v_3$] (z7) at  (2.25,-1) {};
\node[vertex1_style, label=below:\small$v_4$] (z8) at  (2.95,-1) {};
\node[vertex1_style, label=right:\small$v_5$] (z9) at  (3.35,-.8) {};
\node[vertex1_style, label=right:\small$v_{6}$] (z10) at  (3.7,-.3) {};
 \draw[thick] (z6)--(z7)--(z8)--(z6);
 \draw[thick] (z6)--(z9)--(z10)--(z6);
\node[vertex1_style, label=below:\small$\;v_{7}$] (z11) at  (3.4,.7) {};
\node[vertex1_style, label=above:\small$v_{11}$] (z12) at  (2.4,1.1) {};
\node[vertex1_style, label=above:\small$v_{9}$] (z14) at  (5,2.1) {};
\node[vertex1_style, label=above:\small$v_{10}$] (z15) at  (3.85,1.8) {};
\node[vertex1_style, label=south east:\small$v_{8}$] (z13) at  (4.65,1) {};
 \draw[thick] (z6)--(z11)--(z12);
 \draw[thick] (z11)--(z15)--(z14)--(z13)--(z11)--(z13);
 \draw[thick] (z13)--(z15);
 \draw[thick] (z11)--(z14);
\draw[thick] (z6)--(z12);
		
\end{scope}
   \begin{scope}[ xshift=-10em]
              \node[nonterminal] at (0,1) {$G_1$};

 \node[vertex1_style, label=above:\small$v_1$] (z5) at  (1.5,0) {};

 \node[vertex1_style, label=above:\small$v_2$] (z6) at  (2.6,0) {};
 \draw[thick] (z6)--(z5);
\node[vertex1_style, label=below:\small$v_3$] (z7) at  (2.25,-1) {};
\node[vertex1_style, label=below:\small$v_4$] (z8) at  (2.95,-1) {};
\node[vertex1_style, label=right:\small$v_5$] (z9) at  (3.35,-.8) {};
\node[vertex1_style, label=right:\small$v_{6}$] (z10) at  (3.7,-.3) {};
 \draw[thick] (z6)--(z7)--(z8)--(z6);
 \draw[thick] (z6)--(z9)--(z10)--(z6);
\node[vertex1_style, label=below:\small$\;v_{7}$] (z11) at  (3.4,.7) {};
\node[vertex1_style, label=above:\small$v_{11}$] (z12) at  (2.6,1.4) {};
\node[vertex1_style, label=above:\small$v_{9}$] (z14) at  (5,2.1) {};
\node[vertex1_style, label=above:\small$v_{10}$] (z15) at  (3.85,1.8) {};
\node[vertex1_style, label=south east:\small$v_{8}$] (z13) at  (4.65,1) {};
 \draw[thick] (z6)--(z11)--(z12);
 \draw[thick] (z11)--(z15)--(z14)--(z13)--(z11)--(z13);
 \draw[thick] (z13)--(z15);
 \draw[thick] (z11)--(z14);
\end{scope}

\end{tikzpicture} }
\end{center}
  \caption{ \label{forte}  \small The clique trees considered in Example 1 and Example 2}
  \end{figure}
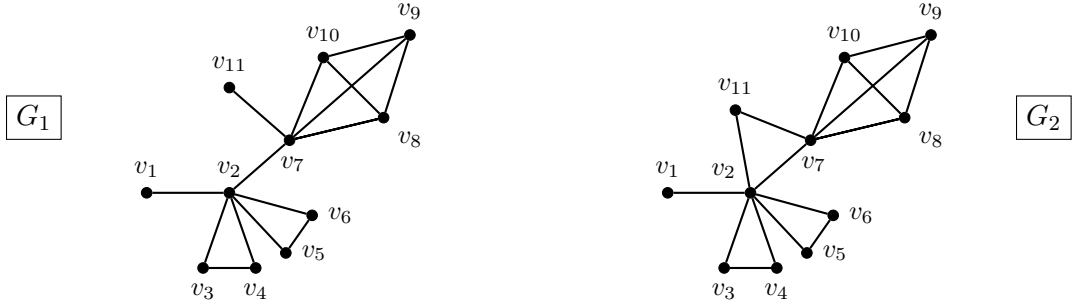
\begin{examplex}\label{ex1} {\rm Let $G_1$ be the graph depicted on the left in Fig. 2. The radius and the diameter of $G_1$ are $2$ and $3$ respectively. Taking into account the vertex-labelling of Fig. 2, $C(G_1)=C(T_{G_1}) = \{v_2, v_7 \}$. As in the proof of Theorem \ref{3.5}, we consider the graph $G'$ obtained from $G_1$ by removing all the edges of the central block $B_C$. In the case at hand $B_C$ just has one edge, namely $v_2v_7$. $T_1$, i.e.\ the component of $G'$ containing $v_2$, is the subgraph induced by $\{v_i \mid 1 \leqslant i \leqslant 6 \}$; whereas $T_2$, i.e.\ the component of $G'$ containing $v_7$ is the subgraph induced by $\{ v_i \mid 7 \leqslant i \leqslant 11\}$. The partition occuring along the proof of Theorem \ref{4.1}, specialized to the case $G=G_1$, is\\
 \[ \begin{array}{lllll}
 W_{1}:=\{u\in V_{T_{1}}  \mid  d(u,v_2)=1\} = V_{T_1} \setminus \{v_2\} & W_{3}:=\{v\in V_{T_{2}} \mid d(v,z_{2})=1\}=V_{T_2} \setminus\{v_7\},\\[.4em]
 W_{2}:=\{u\in V_{T_{1}} \mid 0\leqslant d(u,v_{2})<1\} = \{v_2\}, &
 W_{4}:=\{v \in V_{T_{2}} \mid 0\leqslant d(v,v_{7})<1\}=\{v_7\},
\end{array}
\]
 and
 $W_{5}=\varnothing$. The matrix $\mathcal E(G_1)$, with respect to the vertex ordering
\[ \underbrace{v_1, \, v_3, \,v_4, \,v_5,\, v_6}_{\in W_1}, \underbrace{\,  v_{2} \,}_{\in W_2}, \underbrace{v_8,\,v_9,\,v_{10},\,v_{11} \,}_{\in W_3},\underbrace{\, v_{7} \,}_{\in W_4},
\]
assumes a form of type \eqref{Eq:matrix10}, and can be entirely written down from the two rows
\[ \mathcal E(G_1)_{v_1} = (0,0,0,0,0,0,3,3,3,3,2) \qquad \text{and} \qquad  \mathcal E(G_1)_{v_2} = (0,0,0,0,0,0,2,2,2,2,0).
\]
The characteristic polynomial is $\lambda^7(\lambda^4-216\lambda^2+320)$, and
\[ {\rm Spec}_{\mathcal E} (G_1) = \left\{ 2\sqrt{27+\sqrt{709}},\,2\sqrt{27-\sqrt{709}},0^{(7)},\, -2\sqrt{27-\sqrt{709}},\,
-2\sqrt{27+\sqrt{709}} \right\}. \]
The results are consistent with Theorems \ref{3.5} and \ref{4.1}. In fact, ${\rm In}(\mathcal E(G_1))=(2,2,7)$ and the $\mathcal E$-spectrum of $G_1$ is symmetric.
}
\end{examplex}
\begin{examplex}\label{ex2} {\rm Let $G_2$ be the graph depicted on the right in Fig. 2. The radius and the diameter of $G_2$ are $2$ and $3$ respectively. Taking into account the vertex-labelling of Fig. 2, we have $C(G_2)=\{v_2,v_7,v_{11}\}$ and $ C(T_{G_2})=\{v_2,v_7\}$. As in the proof of Theorem \ref{3.5}, we consider the graph $G'$ obtained from $G_2$ by removing all the edges of the central block $B_C$. In the case at hand $B_C$ just has three edges, namely $v_2v_7$, $v_2v_{11}$ and $v_7v_{11}$. $T_1$, i.e.\ the component of $G'$ containing $v_2$, is the subgraph induced by $\{v_i \mid 1 \leqslant i \leqslant 6 \}$; whereas $T_2$, i.e.\ the component of $G'$ containing $v_7$ is the subgraph induced by $\{ v_i \mid 7 \leqslant i \leqslant 10\}$. The partition occuring along the proof of Theorem \ref{4.1}, specialized to the case $G=G_2$, is\\
 \[ \begin{array}{lllll}
 W_{1}:=\{u\in V_{T_{1}}  \mid  d(u,v_2)=1\} = V_{T_1} \setminus \{v_2\} & W_{3}:=\{v\in V_{T_{2}} \mid d(v,z_{2})=1\}=\{v_8,\,v_9,\,v_{10}\},\\[.4em]
 W_{2}:=\{u\in V_{T_{1}} \mid 0\leqslant d(u,v_{2})<1\} = \{v_2\}, &
 W_{4}:=\{v \in V_{T_{2}} \mid 0\leqslant d(v,v_{7})<1\}=\{v_7\},
\end{array}
\]
 and
 $W_{5}=\{v_{11}\}$. The matrix $\mathcal E(G_2)$, with respect to the vertex ordering
\[ \underbrace{v_1, \, v_3, \,v_4, \,v_5,\, v_6}_{\in W_1}, \underbrace{\,  v_{2} \,}_{\in W_2}, \underbrace{v_8,\,v_9,\,v_{10} \,}_{\in W_3},\underbrace{\, v_{7} \,}_{\in W_4},\underbrace{\, v_{11} \,}_{\in W_5}
\]
is of type \eqref{Eq:matrix11}, and can be entirely written down from the three rows
\[ \mathcal E(G_2)_{v_1} = (0,0,0,0,0,0,3,3,3,2,2) \qquad \text{and} \qquad  \mathcal E(G_2)_{v_2} = (0,0,0,0,0,0,2,2,2,0,0),
\]
and
\[ \mathcal E(G_2)_{v_8} = (3,3,3,3,3,2,0,0,0,0,2).
\]
The characteristic polynomial is $\lambda^7(\lambda^4-199\lambda^2-360\lambda+720)$, and
\[ {\rm Spec}_{\mathcal E} (G_2) = \left\{14.8323,\,1.2042, \,0^{(7)},\, -3.1211,\, -12.9155 \right\} \]
(the value of the nonzero eigenvalues is approximated).
The results are consistent with Theorems \ref{3.5} and \ref{4.1}. In fact, ${\rm In}(\mathcal E(G_2))=(2,2,7)$ and the $\mathcal E$-spectrum of $G_2$ is not symmetric.

}
\end{examplex}
\begin{examplex}\label{ex3} {\rm Let $G_3$ be the graph depicted on the left in Fig. 1. The radius and the diameter of $G_3$ are $2$ and $4$ respectively. Taking into account the vertex-labelling of Fig. 1, the only central vertex of $G_3$ is $v_6$. The set of diametrally distinguished vertices is $\mathcal D_{G_3}= \{v_5,v_{11}\}$. According to the definitions given along the proof of Theorem \ref{3.6}, the graphs $T_1$, $T_2$ are $T_3$ are the subgraph of $G_3$ respectively induced by $\{ v_i \mid 1 \leqslant i \leqslant 5\}$, $\{ v_i \mid 11 \leqslant i \leqslant 15\}$ and $\{ v_i \mid 6 \leqslant i \leqslant 10\}$. Moreover,
\[ \begin{array}{ @{}  ll @{} }
      V_{1}=\{u\in T_{1} \mid d(u,v_6)=2\} =\{ v_1,v_2,v_3,v_4\} & V_{2}=\{u\in T_{2} \mid d(u,v_6)=2\} =\{ v_{12},v_{13},v_{14},v_{15}\}  \\[1.5\jot]
      V_{3}=\{u\in T_{1} \mid 0\leqslant d(u,v_5)<1\}=\{v_5\} &  V_{4}=\{u\in T_{2} \mid 0\leqslant d(u,v_{11})<1\}=\{v_{11}\} \\[1.5\jot]
V_5=V_{T_3}=\{ v_i \mid 6 \leqslant i \leqslant 10\}.
    \end{array}
\]
The matrix $\mathcal E(G)$, chosen the vertex ordering
\[ \underbrace{v_1, \, v_2, \,v_3, \,v_4}_{\in V_1}, \underbrace{ v_{12}, \,v_{13}, \,v_{14}, \,v_{15}}_{\in V_2}, \underbrace{\, v_5 \,}_{\in V_3},\underbrace{\, v_{11} \,}_{\in V_4}, \underbrace{ v_{6}, \,v_{7}, \,v_{8}, \,v_{9},\, v_{10}}_{\in V_5},
\]
assumes the form \eqref{Eq:matrix1}, and can be easily written down from the following data: $k=2$ and the rows corresponding to the vertices $v_1 \in V_1$ and $v_{12} \in V_2$ are
\[(0,0,0,0,4,4,4,4,0,3,2,3,3,3,3)
\qquad
\text{and}
\qquad (4,4,4,4,0,0,0,0,3,0,2,3,3,3,3).\]
It turns out that the characteristic polynomial of $\mathcal E(G_3)$ is $(x-2)(x+18)(x^2-16x-356)$. Therefore,
\[ {\rm Spec}_{\mathcal E}(G_3)= \left\{8+2\sqrt{105}, \, 2, \, 0^{(11)}, \, 8-2\sqrt{105}, \, -18\right\}, \qquad \text{and} \qquad  {\rm In}(\mathcal E(G_3)=(2,2,11),\]
as expected after Theorem \ref{3.6}. Clearly, ${\rm Spec}_{\mathcal E}(G_3)$ is not symmetric, consistently with Theorem \ref{4.5}.
}
\end{examplex}

\section{Further Discussions}

In this paper we just consider the set $\mathcal C  \mathcal T$ of clique trees whose blocks have at most two cut-vertices of the clique tree. Otherwise, when the clique trees whose at least one block has more than three cut-vertices, then the inertia of eccentricity matrices no longer has the above characteristics. Let's use the example of Figure 3 to illustrate this point.\\\\
\unitlength 1mm 
\linethickness{0.9pt}
\ifx\plotpoint\undefined\newsavebox{\plotpoint}\fi 
\begin{picture}(101.923,26.04)(-25,0)
\put(.274,11.045){\makebox(0,0)[cc]{$v_{1}$}}
\put(.613,14.94){\circle*{1.5}}
\put(.457,15.091){\line(1,-1){5.303}}
\put(5.858,5.722){\makebox(0,0)[cc]{$v_{2}$}}
\put(6.197,9.617){\circle*{1.5}}
\put(17.157,20.831){\circle*{1.5}}
\put(18.856,22.32){\makebox(0,0)[cc]{$v_{5}$}}
\put(12.384,25.427){\circle*{1.5}}
\put(10.834,11.909){\makebox(0,0)[cc]{$v_{4}$}}
\put(11.323,14.912){\circle*{1.5}}
\put(22.377,11.732){\makebox(0,0)[cc]{$v_{6}$}}
\put(23.014,14.883){\circle*{1.5}}
\put(17.098,20.643){\line(-1,-1){11.125}}
\put(20.373,0){\makebox(0,0)[cc]{$H_{1}$}}
\put(37.834,14.969){\circle*{1.5}}
\put(39.954,11.623){\makebox(0,0)[cc]{$v_{9}$}}
\put(.805,14.982){\line(1,0){37.123}}
\multiput(12.119,25.412)(.0337360106,-.0337360106){524}{\line(0,-1){.0337360106}}
\multiput(23.079,14.982)(.033634504,.033634504){226}{\line(1,0){.033634504}}
\put(8.23,23.644){\makebox(0,0)[cc]{$v_{3}$}}
\put(30.327,21.877){\line(0,-1){14.142}}
\multiput(37.575,14.805)(-.033671751,-.033671751){210}{\line(-1,0){.033671751}}
\put(30.507,21.926){\circle*{1.5}}
\put(34.098,24.466){\makebox(0,0)[cc]{$v_{7}$}}
\put(30.359,7.737){\circle*{1.5}}
\put(34.418,6.596){\makebox(0,0)[cc]{$v_{8}$}}
\multiput(30.786,21.908)(.033714635,-.033714635){194}{\line(0,-1){.033714635}}
\put(84.011,19.344){\circle*{1.5}}
\put(98.186,19.487){\circle*{1.5}}
\put(81.536,14.494){\circle*{1.5}}
\put(71.411,11.869){\circle*{1.5}}
\multiput(71.423,11.756)(.13141026,.03365385){78}{\line(1,0){.13141026}}
\multiput(81.298,14.756)(-.033707865,-.045646067){178}{\line(0,-1){.045646067}}
\multiput(71.548,11.506)(.033613445,-.044117647){119}{\line(0,-1){.044117647}}
\put(75.411,6.494){\circle*{1.5}}
\multiput(83.923,19.256)(-.03365385,-.06089744){78}{\line(0,-1){.06089744}}
\multiput(98.173,19.756)(.03373016,.06349206){63}{\line(0,1){.06349206}}
\put(100.036,23.744){\circle*{1.5}}
\put(84.173,19.256){\line(1,0){14}}
\multiput(83.673,19.631)(.138655462,.033613445){119}{\line(1,0){.138655462}}
\put(94.036,14.994){\circle*{1.5}}
\put(91.786,24.119){\circle*{1.5}}
\put(91.798,23.881){\line(1,0){8.25}}
\put(76.661,17.869){\circle*{1.5}}
\multiput(76.923,17.881)(.14540816,.03316327){49}{\line(1,0){.14540816}}
\put(67.673,9.381){\makebox(0,0)[cc]{$v_{1}$}}
\put(71.548,4.881){\makebox(0,0)[cc]{$v_{2}$}}
\put(73.673,16.506){\makebox(0,0)[cc]{$v_{3}$}}
\put(81.548,10.506){\makebox(0,0)[cc]{$v_{4}$}}
\put(82.048,20.131){\makebox(0,0)[cc]{$v_{5}$}}
\put(91.798,12.131){\makebox(0,0)[cc]{$v_{7}$}}
\put(88.298,23.256){\makebox(0,0)[cc]{$v_{6}$}}
\put(101.923,21.506){\makebox(0,0)[cc]{$v_{8}$}}
\put(100.923,17.131){\makebox(0,0)[cc]{$v_{9}$}}
\put(86.107,.42){\makebox(0,0)[cc]{$H_{2}$}}
\multiput(81.516,14.723)(.125,.03358209){134}{\line(1,0){.125}}
\multiput(81.766,14.848)(.0701923077,.0336538462){260}{\line(1,0){.0701923077}}
\multiput(97.547,18.924)(-.033650943,-.033660377){106}{\line(0,-1){.033660377}}
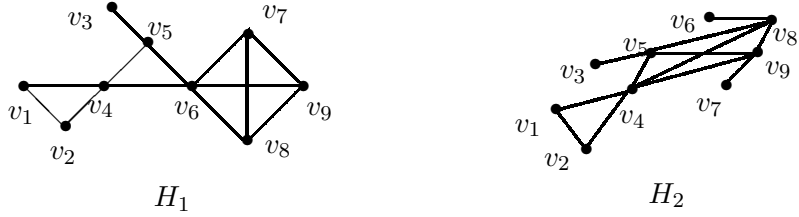
\put(60.5,-8){\makebox(0,0)[cc]{Figure 3: $H_{1}$ ($H_{2}$) have a block which exist three (four) cut-vertices of the clique tree.}}
\end{picture}

\vspace{1.8cm}
Obviously, $H_{1}$ has one block indexed by $\{v_{4},v_{5},v_{6}\}$ which has three cut-vertices of the clique tree and $H_{2}$ has one block indexed by $\{v_{4},v_{5},v_{8},v_{9}\}$ which has three cut-vertices of the clique tree, then we can give their eccentricity matrices and eccentricity spectrum.

$$\mathcal E(H_{1})=
    \begin{pmatrix}
0   & 0      & 3       &0      &2      &2      &3      &3        &3     \\
0   & 0      & 3       &0      &2      &2      &3      &3        &3     \\
3   & 3      & 0       &2      &0      &2      &3      &3        &3      \\
0   & 0      & 2       &0      &0      &0      &2      &2        &2      \\
2   & 2      & 0       &0      &0      &0      &2      &2        &2      \\
2   & 2      & 2       &0      &0      &0      &0      &0        &0     \\
3   & 3      & 3       &2      &2      &0      &0      &0        &0     \\
3   & 3      & 3       &2      &2      &0      &0      &0        &0     \\
3   & 3      & 3       &2      &2      &0      &0      &0        &0
    \end{pmatrix}
 ,   \mathcal E(H_{2})=\begin{pmatrix}
0   & 0      & 3       &0      &2      &3      &3      &2        &2     \\
0   & 0      & 3       &0      &2      &3      &3      &2        &2     \\
3   & 3      & 0       &2      &0      &3      &3      &2        &2     \\
0   & 0      & 2       &0      &0      &2      &2      &0        &0     \\
2   & 2      & 0       &0      &0      &2      &2      &0        &0     \\
3   & 3      & 3       &2      &2      &0      &3      &0        &2     \\
3   & 3      & 3       &2      &2      &3      &0      &2        &0     \\
2   & 2      & 2       &0      &0      &0      &2      &0        &0     \\
2   & 2      & 2       &0      &0      &2      &0      &0        &0
    \end{pmatrix}
$$
It turns out that the characteristic polynomial of $\mathcal E(H_1)$ is $-x^3 (-1536 + 1728 x + 528 x^2 - 588 x^3 - 147 x^4 + x^6)$. Therefore,\\
\[ {\rm Spec}_{\mathcal E}(H_1)= \left\{-8.65993, \,-5.03793, \,-2.15288,\, 0^{(3)}, \,1.03969,\, 1.1515, \, 13.6595\right\}\]

And the characteristic polynomial of $\mathcal E(H_1)$ is $-(-1 + x)^2 x (4 + x)^2 (288 - 144 x - 106 x^2 - 6 x^3 + x^4)$. Therefore,\\
\[ {\rm Spec}_{\mathcal E}(H_1)= \left\{-6.17059, \,-4^{(2)}, \,-3.04056,\, 0, \,1^{(2)},\, 1.08678, \, 14.1244\right\}\]\\


{\footnotesize \begin{thebibliography}{12}
\bibitem{Band} H.-J. Bandelt and H. M. Mulder,  Distance-hereditary graphs,  J. Comb. Theory Ser. B, 41 (2) (1986), 182--208.

\bibitem{JB}
J. R. S. Blair and B. Peyton, An introduction to chordal graphs and clique trees, Graph Theory and Sparse Matrix Computation, IMA Vol. Math. Appl., 56 (1993), 1--29.

\bibitem{BH} A. E. Brouwer and W. H. Haemers, Spectra of Graphs, Springer Universitext, New York 2012.

\bibitem{D} R. Diestel, Graph Theory, Graduate Texts in Mathematics vol. 173, 5th edition, Springer-Verlag, Heidelberg (2017).
\bibitem{Har} F. Harary, Graph Theory, Addison-Wesley Publishing Company, Reading (Massachusetts) (1969).


\bibitem{RHC}
R. A. Horn and C. R. Johnson. Matrix analysis. Cambridge University Press, Cambridge, second edition, 2013.
\bibitem{How} E. Howorka,
On metric properties of certain clique graphs,
J. Comb. Theory Ser. B, 27 (1) (1979), 67--74.
\bibitem{JGZ}
 Y. L. Jin, R. Gu and X. D. Zhang, The distance energy of clique trees, Linear Algebra Appl., 615 (2021), 1--10.
\bibitem{KM} P. S. Kumar and C. E. V. Madhavan, Clique tree generalization and new subclasses of chordal graphs, Discrete Appl. Math.,  117 (1-3) (2002), 109--131.
\bibitem{LX}
X. Y. Lei, J. F. Wang and G. Z. Li, On the eigenvalues of eccentricity matrix of graphs, Discrete Appl. Math., 295 (2021), 134--147.
\bibitem{LLL}
H. Q. Lin, R. F. Liu and X. Y. Lu, The inertia and energy of the distance matrix of a connected graph, Linear Algebra Appl., 467 (2015), 29--39.


\bibitem{IM}
I. Mahato, R. Gurusamy, M. R. Kannan and S. Arockiaraj, Spectra of eccentricity matrices of graphs, Discrete Appl. Math., 285 (2020), 252--260.
\bibitem{MK}
I. Mahato and M. R.Kannan, On the eccentricity matrices of trees: Inertia and spectral symmetry, Discrete Math., 345 (2022)113067.


\bibitem{AKP}
A. K. Patel, L. Selvaganesh and S. K. Pandey, Energy and inertia of the eccentricity matrix of coalescence of graphs. Discrete Math., 344 (12) (2021)112591.
\bibitem{ZZ}
Z. P. Qiu and Z. K. Tang, On the eccentricity spectra of threshold graphs, Discrete Appl. Math., 310 (2022), 75--85.
\bibitem{MR}
M. Randi\'{c}, $D_{MAX}$-matrix of dominant distances in a graph, MATCH Commun. Math. Comput. Chem., 70 (2013), 221--238.
\bibitem{MRA}
M. Randi\'{c}, R. Orel and A. T. Balaban, $D_{MAX}-$matrix invariants as graph descriptors. Graphs having the same Balaban index $J$, MATCH Commun. Math. Comput. Chem., 70 (2013), 239--258.
\bibitem{JX}
J. F. Wang, X. Y. Lei, S. C. Li, W. Wei and X. B. Luo, On the eccentricity matrix of graphs and its applications to the boiling point of hydrocarbons, Chemometrics and Intelligent Laboratory Systems, 207 (2020), 104--173.
\bibitem{JG}
J. F. Wang, L. Lu, M. Randi\'{c} and G. Z. Li, Graph energy based on the eccentricity matrix, Discrete Math., 342 (9) (2019),  2636--2646.
\bibitem{JM}
J. F. Wang, M. Lu, F. Belardo and M. Randi\'{c}, The anti-adjacency matrix of a graph: Eccentricity matrix, Discrete Appl. Math., 251 (2018), 299--309.
\bibitem{WB} J. F. Wang, M. Lu, M. Brunetti, L. Lu and X. Huang, Spectral determinations and eccentricity matrix of graphs, Adv. Appl. Math. 139 (2022)102358.
\bibitem{JMF}
J. F. Wang, M. Lu, L. Lu and F. Belardo, Spectral properties of the eccentricity matrix of graphs, Discrete Appl. Math., 279 (2020), 168--177.




\bibitem{WW}
W. Wei, X. C. He and S.C Li, Solutions for two conjectures on the eigenvalues of the eccentricity matrix, and beyond, Discrete Math., 343 (8) (2020), 111-925.





\bibitem{XJ}
X. L. Zhang and J. J. Zhou, The Distance Laplacian Spectral Radius of Clique Trees, Discrete Dynamics in Nature and Society, 2020 (2020)8855987.

\end{thebibliography}}
\end{document}